\documentclass{amsart}
\usepackage{lineno,hyperref,paralist,graphicx,epstopdf,float,bookmark,geometry,amssymb,amsmath,amsthm}
\newtheorem{The}{Theorem}[section]
\newtheorem{Cor}[The]{Corollary}
\newtheorem{Lem}[The]{Lemma}
\newtheorem{Pro}[The]{Proposition}
\theoremstyle{definition}
\newtheorem{defn}[The]{Definition}
\newtheorem*{Rem}{Remark}
\numberwithin{equation}{section}
\newcommand{\h}{\mathbf{H}}
\newcommand{\J}{\mathcal{J}}
\newcommand{\R}{\mathbb{R}}
\newcommand{\Z}{\mathbb{Z}}

\newcommand{\s}{\mathbb{S}^1}
\newcommand{\Q}{\mathbb{Q}}
\newcommand{\M}{\mathrm{M}}
\newcommand{\m}{\mathfrak{M}}
\newcommand{\A}{\mathrm{A}}
\newcommand{\B}{\mathbf{B}}
\newcommand{\an}{\mathbf{A}}
\newcommand{\D}{\mathbf{D}}
\newcommand{\E}{\mathbf{E}}

\begin{document}
\title{Regular dependence of the Peierls barriers on perturbations}
\author{Qinbo Chen$^{\dag}$}
\address{$^{\dag}$\; Department of Mathematics, Nanjing University\\Nanjing, Jiangsu, China, 210093}
\email{qinboChen1990@gmail.com}
\author{Chong-Qing Cheng$^\ddag$}
\address{$^\ddag$\; Department of Mathematics, Nanjing University\\Nanjing, Jiangsu, China, 210093}
\email{chengcq@nju.edu.cn}
\begin{abstract}
Let $f$ be an exact area-preserving  monotone twist diffeomorphism of the infinite cylinder and $P_{\omega,f}(\xi)$ be the associated Peierls barrier.
In this paper, we give the  H\"{o}lder regularity of $P_{\omega,f}(\xi)$ with respect to the parameter $f$. In fact, we prove that if the rotation
symbol $\omega\in (\R\setminus\Q)\bigcup(\Q+)\bigcup(\Q-)$, then $P_{\omega,f}(\xi)$ is $1/3$-H\"{o}lder continuous in $f$, i.e.
$$|P_{\omega,f'}(\xi)-P_{\omega,f}(\xi)|\leq C\|f'-f\|_{C^1}^{1/3} ,~~\forall \xi\in\R$$
where $C$ is a constant. Similar results also hold for the Lagrangians with one and a half degrees of freedom. As application, we give an open and dense result about the breakup of invariant circles.
\end{abstract}
\maketitle

\section{Introduction}
The Peierls barrier $P_{\omega,f}(\xi)$ for the monotone twist diffeomorphism $f$ is a function which can be thought of as a dislocation energy.
It measures to which extent the stationary configuration $(x_i)_{i\in\Z}$ of rotation symbol $\omega$, subject to the constraint $x_0=\xi$, is
not minimal.  In \cite{MR920622}, Mather established the modulus of continuity for $P_{\omega,f}$  with respect to the parameter $\omega$, and
by applying this property, he gave destruction results for invariant circles under arbitrary small perturbations (\cite{MR967638}). For further
research, we need more information and properties about the Peierls barriers. In this paper,  we prove the H\"{o}lder continuity of {P}eierls
barriers with respect to the parameter $f$, which generalize J. Mather's results in \cite{MR920622}\cite{MR967638}. This paper is organized as follows:
In Section 2 and Section 3, we introduce the definition of Peierls barrier and  some basic properties in Aubry-Mather theory. Our main results
are Theorem \ref{main theorem 1} and Theorem \ref{main theorem 2} in Section 4. In Section 5, we give an open and dense result as an application example.

For convenience, we denote by $\vartheta$ (mod 1) the standard coordinate of $\s=\R/\Z$ and $x$ the corresponding coordinate of its universal cover $\R$. We will let $(\vartheta,y)$ denote the standard coordinates of $\s\times\R$ and $(x,y)$ the  corresponding coordinates of the universal cover $\R\times\R$.
The dynamical properties of exact area-preserving monotone twist diffeomorphisms of an infinite cylinder $\s\times\R$ have been studied by Mather (\cite{MR670747}--\cite{MR1139556}) and by Bangert (\cite{MR945963}). In the following, we refer to these papers for the definitions and results that we'll need.
\subsection{Monotone twist diffeomorphism}
\begin{defn}
We call $f$ an exact area-preserving monotone twist diffeomorphism if $f: \s\times\R \longrightarrow \s\times\R$,
$$(\vartheta,y)\longmapsto(\vartheta',y')$$
is a diffeomorphism  satisfying the following conditions:
\begin{enumerate}[(1)]
  \item $f\in C^1(\s\times\R)$.
  \item The 1-form $y'd\vartheta'-yd\vartheta$ on $\s\times\R$ is exact.
  \item (positive monotone twist) $\frac{\partial\vartheta'(\vartheta,y)}{\partial y}>0$, for all $(\vartheta,y)$.
  \item $f$ twists the cylinder infinitely at either hand. To express this condition, we consider a lift $\bar{f}$ of $f$ to the universal cover $\R\times\R$, $\bar{f}(x,y)=(x',y'),$ the condition means that for fixed $x$,
      $$x'\rightarrow +\infty ~\text{as}~ y\rightarrow +\infty~ \text{and}~ x'\rightarrow -\infty ~\text{as}~ y\rightarrow -\infty.$$

\end{enumerate}
\end{defn}

The positive monotone twist condition has its geometrical meaning. Consider a point $P\in\s\times\R$ and denote by $v_P=(0,1)$ the vertical vector at $P$. Let $\beta_f(P)$ denote the angle between $v_P$ and  $d_Pf \cdot v_P$ (count in the clockwise direction). So the  positive monotone twist condition means that
$$0<\beta_f(P)<\pi$$
everywhere.
\begin{figure}[H]
  \centering
  \includegraphics[width=5cm]{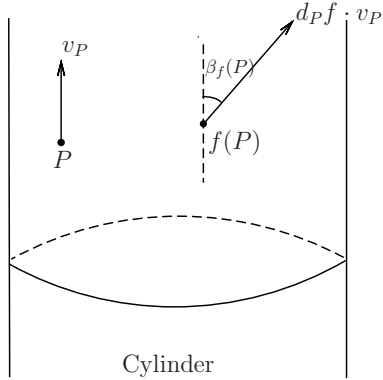}\\
  \caption{geometrical meaning of positive monotone twist condition}
\end{figure}

We denote by $\J$ the class of exact area-preserving monotone twist diffeomorphisms. Let $\J_\beta=\{f\in\J: \beta_f(P)\geq\beta,~\text{for all}~P\in\s\times\R\}$. Although $\bigcup\limits_{\beta> 0}\J_\beta \subsetneqq \J$, most of our results can be generalized to $\J$ without any difficulty. This is because our main results concern what happens in a compact region $K$ of $\s\times\R$. Thus, for all $f\in\J $, there exists $\beta>0$ and $g\in\J_\beta$ such that $f\mid_K=g\mid_K$.
\subsection{The variational principle}
If $f\in\J$ and $\bar{f}$ is a lift of $f$ to $\R\times\R$ such that $\bar{f}(x,y)=(x',y')$, then there exists a $C^2$ generating function $h(x,x'):\R\times\R\longrightarrow\R$ such that $\partial_{12}h:=\frac{\partial^2h(x,x')}{\partial x\partial x'}<0$ and
\begin{equation}\label{generating}
\begin{cases}
y=-\partial_1h(x,x')\\
y'=\partial_2h(x,x')
\end{cases}
\end{equation}
where $\partial_1h(x,x')$ and $\partial_2h(x,x')$ denote the partial differential derivatives of $h$ with respect to $x$ and $x'$.

For $f\in\J$, the generating function satisfies the following conditions $(\h1)-(\h4)$,\\
($\h1$) $h(x,x')=h(x+1,x'+1)$, for all $x,x'\in\R.$\\
($\h2$) $\lim\limits_{|\xi|\rightarrow+\infty}h(x,x+\xi)=+\infty$ uniformly in $x$.\\
($\h3$) If $x<\xi, x'<\xi'$, then $h(x,x')+h(\xi,\xi')<h(x,\xi')+h(\xi,x').$\\
($\h4$) If $(\bar{x},x,x')$ and $(\bar{\xi},x,\xi')$ are both minimal segments (see \S2) and are distinct, then
  $$(\bar{x}-\bar{\xi})(x'-\xi')<0.$$

Moreover, we add two further conditions to $h$ which was firstly introduced by Mather.\\
($\h5$) There exists a positive continuous function $\rho:\R^2\rightarrow\R$ such that for $x<\xi$ and $x'<\xi'$,
  $h(x,\xi')+h(\xi,x')-h(x,x')-h(\xi,\xi')\geq \int^\xi_x \int^{\xi'}_{x'}\rho(s,s')dsds'.$\\
($\h6\theta$) There exists a positive number $\theta$ such that
$$x\mapsto\theta(x-x')^2/2-h(x,x') ~\text{is convex},~\text{for any}~ x',$$
$$x'\mapsto\theta(x'-x)^2/2-h(x,x') ~\text{is convex},~\text{for any}~ x.$$

Conditions $(\h1)-(\h4)$ were firstly introduced by Bangert in \cite{MR945963}, and conditions $(\h3)(\h4)$ can be implied by conditions $(\h5)(\h6)$ (see \cite{MR920622}). For $f\in\J_\beta$, it's not hard to observe that we can take $\rho=-\partial_{12}h$ in $(\h5)$, the inequality ``$\geq$'' can be replaced by ``='', and $(\h6\theta)$ is satisfied with $\theta=\cot\beta~(0<\beta<\pi/2).$

Notice that if $h$ is a generating function, $h+C$ is still a generating function. In next section, we'll introduce some basic results on the theory of minimal configurations, which was developed by Aubry and Le Daeron \cite{MR719634} and Mather. However, Bangert generalized this theory where $h$ is not necessary differentiable and only satisfies $(\h1)-(\h4)$ (see \cite{MR945963}).

\section{Minimal configurations and Peierls barriers}
\subsection{Minimal configurations}
In Bangert's set up, the variational principle $h$ need not be differentiable, which is very useful for us in our proof of Theorem \ref{main theorem 1} and Theorem \ref{main theorem 2}. Therefore, in this section, unless otherwise specified, we assume that $h$ is only continuous and satisfies the conditions $(\h1)-(\h4)$.
\begin{defn}
Let $\R^Z=\{x|x:\Z\rightarrow\R\}$ be  bi-infinite sequences of real numbers with the product topology and let $x=(x_i)_{i\in\Z}$ be an element in $\R^Z$. We extend $h$ to a finite segment $(x_j,...,x_k), j<k$,
$h(x_j,...,x_k):=\sum\limits_{i=j}^k h(x_i,x_{i+1})$. The segment $(x_j,...,x_k)$ is called \emph{ minimal segment} with respect to $h$ if
$$h(x_j,...,x_k)\leq h(y_j,...,y_k),$$
for all $(y_j,...,y_k)$ with $y_j=x_j$ and $y_k=x_k$. $x=(x_i)_{i\in\Z}$ is called \emph{minimal configuration} if every finite segment of $x$ is minimal.
\end{defn}

We denote by $\M:=\M_h$ the set of all minimal configurations of $h$.

\begin{defn}
If $h\in C^2$, a segment $(x_j,...,x_k)$ is called \emph{stationary }if
$$\partial_2h(x_{i-1},x_i)+\partial_1h(x_i,x_{i+1})=0,~\text{for all}~j<i<k.$$
\end{defn}
\begin{Rem}
The stationary configurations $(...,x_i,...)$ of $h$ correspond to the orbits $(...,(x_i,y_i),...)$ of $\bar{f}$,
where $y_i=-\partial_1h(x_i,x_{i+1})=\partial_2h(x_{i-1},x_i)$.
\end{Rem}
Thus, the minimal configurations $\M_h$ of $h$ correspond to a class of minimal orbits of the lift $\bar{f}$ of $f$,  we denote by $\m:=\m_{\bar{f}}$ the set of all such minimal orbits. In the following, we will give some dynamical properties for the minimal configurations for $h$, or equivalently, the minimal orbits of $\bar{f}$. For proofs, we refer to \cite{MR920622}.

Given a configuration $x=(x_i)_{i\in\Z}$, we join $(i,x_i)$ and $(i+1,x_{i+1})$ by a line segment in $\R^2$, the union of all such line segments is a piecewise linear curve in $\R^2$, which we call \emph{the Aubry graph} of $x$. We say that configurations $x$ and $x^*$ \emph{cross} if their Aubry graph cross. We say $x<x^*$ if $x_i<x^*_i$ for all $i$. Similarly, we can define $x>x^*$ and $x=x^*$. We say that $x$ and $x^*$ are \emph{comparable} if $x<x^*$ or $x=x^*$ or $x>x^*$. By condition $(\h4)$, we know that any two minimal configurations either cross or are comparable.
Moreover, we say $x$ and $x^*$ are \emph{$\omega-asymptotic$} (resp. \emph{$\alpha-asymptotic$}) if $\lim\limits_{i\rightarrow+\infty}|x_i-x^*_i|=0$ (resp. $\lim\limits_{i\rightarrow-\infty}|x_i-x^*_i|=0$).

\begin{Lem}[\textup{\cite{MR945963}}, Aubry's Crossing Lemma]\label{Aubry crossing lemma}
Let $x$ and $x^*$ be $h$-minimal configurations, then they cross at most once. If $x$ and $x^*$ coincide at some $i\in\Z$, i.e. $x_i=x_i^*$, then they cross at $i$.
\end{Lem}

Notice that Aubry's Crossing Lemma can be proved by condition $(\h3)$. If $x$ is a minimal configuration, then
$$\rho(x):=\lim\limits_{n\rightarrow +\infty}x_n/n$$ exists. The number $\rho(x)$ is called \emph{the rotation number} of $x$.
In addition, the rotation function $\rho:\M_h\rightarrow\R$ is continuous and surjective, and $(pr_0,\rho):\M_h\rightarrow \R\times\R$ is proper,
where $pr_0(x)=x_0$.

If $x\in\M_h$ and $\rho(x)=p/q$, where $q>0$ and $p,q$ are relatively prime integers, then $x$ must satisfy one of the three relations (see \cite{MR920622}):
\begin{enumerate}[(a)]
  \item $x_{i+q}>x_i+p,$ for all $i$.
  \item $x_{i+q}=x_i+p,$ for all $i$.
  \item $x_{i+q}<x_i+p,$ for all $i$.
\end{enumerate}
This leads us to introduce the \emph{symbol space} $S=(\R\setminus\Q)\bigcup(\Q-)\bigcup(\Q)\bigcup(\Q+)$, where $\Q+$ denotes the set of all symbols $\frac{p}{q}+$ and $\Q-$ is defined similarly. The symbol space has an obvious order so that $\frac{p}{q}-<\frac{p}{q}<\frac{p}{q}+$. We provide $S$ with the order topology, i.e. the set of intervals ($s_1, s_2$)=$\{x: s_1<x<s_2\}$ is a basis for this topology. We also define the projection map $\pi:S\rightarrow\R$,
\begin{equation}\label{underlying number}
  \pi(\omega)\triangleq
  \begin{cases}
  \omega , & \omega\in(\R\setminus\Q)\\
  \frac{p}{q} , & \omega=\frac{p}{q}\pm ~or ~~\frac{p}{q}
  \end{cases}
\end{equation}
Obviously, the map $\pi$ is weakly order preserving. For more details, see (\cite{MR920622}, \S 3).

From now on, if $\omega\in\R\setminus\Q$, we denote by $\M_\omega:=\M_{\omega,h}$ the set of $x\in\M_h$ with $\rho(x)=\omega$. Let $\M_{p/q}:=\M_{p/q,h}$ denotes the set of $x\in\M_h$ with $\rho(x)=p/q$ and $(b)$ holds. Let $\M_{p/q^-}:=\M_{p/q^-,h}$ denotes the set of $x\in\M_h$ with $\rho(x)=p/q$ and $(b)$ or $(c)$ hold. Similarly, we also denote by $\M_{p/q^+}:=\M_{p/q^+,h}$ the set of $x\in\M_h$ with $\rho(x)=p/q$ and $(a)$ or $(b)$ hold. Notice that for each $\omega\in S$, $\M_{\omega,h}$ is a non-empty closed set and totally ordered.

Equivalently, for the rotation symbol $\omega\in S$, we can denote by $\m_\omega:=\m_{\omega,\bar{f}}$ the set of all $\omega$-minimal orbits of $f\in\J$.
\begin{Pro}\textup{(\cite{MR945963})}\label{rotation estimates}
Let $x=(x_i)_{i\in\Z}\in \M_{\omega,h}$ be a minimal configuration, we have that
$$|x_{i+j}-x_{i}-j\pi(\omega)|<1$$
for all $i<i+j\in\Z$.
\end{Pro}

\begin{Pro}\textup{(\cite{MR1384394})}\label{bangert}
Let $(y_0,...,y_n)$ be a minimal segment, there exists $\alpha\in\R$ so that
$$|y_{i+j}-y_i-j\alpha|<2$$
for all $0\leq i\leq i+j\leq n.$
\end{Pro}

\subsection{Peierls barriers}
By the totally ordered property, the projection $pr_0: \M_{\omega,h}\rightarrow \R$, $pr_0(x)=x_0$ is a homeomorphism of $\M_{\omega,h}$ onto its image. We denote by $\A_{\omega,h}=pr_0(\M_{\omega,h})$, then it's a closed subset of $\R$ and invariant under the translation $x\mapsto x+1$ (see \cite {MR920622}). Now, we begin to introduce the definition of the Peierls barrier function $P_\omega(\xi):=P_{\omega,h}(\xi)$, it measures to which extent the stationary configuration $(y_i)_{i\in\Z}$, subject to the condition $y_0=\xi$, is not minimal.

Besides Bangert's conditions $(\h1)-(\h4)$, we will assume $h$ satisfies $(\h5)$ and $(\h6\theta)$.

Since $\R\setminus \A_{\omega,h}$ is an open set, it is a union of open intervals  $(J_-,J_+)$, where $J_-, J_+\in \A_{\omega,h}$.  \emph{The Peierls barrier }is defined as follows. Let $\omega\in S$ and $\xi\in \R$, in the case $\xi\in \A_{\omega,h}$, we define $P_\omega(\xi)=0$.

In the case $\xi\notin\A_{\omega,h}$, $\xi$ belongs to a complementary interval $(J_-,J_+)$ to $\A_{\omega,h}$. Denote by $x^{\pm}$ the minimal configurations of rotation symbol $\omega$ satisfying $x_0^{\pm}=J_{\pm}$ and let
\begin{equation*}
  I=\begin{cases}
  \Z & \omega\in (\R\setminus\Q)\bigcup(\Q+)\bigcup(\Q-)\\
  {0,...,q-1} & \omega=p/q ,
  \end{cases}
\end{equation*}
 we define
 \begin{equation}\label{the definiton of barrier}
 P_\omega(\xi)=\min\{ \sum\limits_Ih(y_i,y_{i+1})-h(x_i^-,x_{i+1}^-)| y_0=\xi \},
 \end{equation}
where the minimum is taken over the set of all configurations satisfying $x_i^-\leq y_i\leq x_i^+, \forall i\in\Z$ and in the case $\omega=p/q$, the minimum is taken under the additional periodicity constraint $y_{i+q}=y_i+p$.

Notice that $P_\omega(\xi)$ is well defined and finite, since $h$ is Lipschitz and
$$0\leq\sum\limits_{i\in I} y_i-x_i^-\leq\sum\limits_{i\in I} x_i^+-x_i^-\leq 1$$(see \cite{MR945963}).
It is worth pointing out that $P_\omega(\xi)$ is non-negative and 1-periodic, i.e.
$$P_\omega(\xi+1)=P_\omega(\xi).$$

Now, let's recall some important properties of the Peierls barrier which are important for our main results in Section 4.

\begin{Pro}\textup{ (\cite{MR920622},\cite{MR967638})}\label{continuous of modulus}
 Let $h$ be a continuous real valued function satisfying $(\h1)-(\h5)$ and $(\h6\theta)$, then
\begin{enumerate}[(1)]
  \item $|P_{\frac{p}{q}}(\xi)-P_\omega(\xi)|\leq 1200\theta(\frac{1}{q}+|\pi(\omega)q-p|),$
  \item $|P_{\frac{p}{q}+}(\xi)-P_\omega(\xi)|\leq 4800\theta(|\pi(\omega)q-p|)$ in the case $\omega\geq\frac{p}{q}+$ and \\
  $|P_{\frac{p}{q}-}(\xi)-P_\omega(\xi)|\leq 4800\theta(|\pi(\omega)q-p|)$ in the case $\omega\leq\frac{p}{q}-$,
\end{enumerate}
where $\pi$ is (\ref{underlying number}).
\end{Pro}

The proof strongly relies on Aubry's Crossing Lemma. Proposition \ref{continuous of modulus} (1) was firstly proved by Mather in (\cite{MR920622}, Theorem 7.1). The proof of (2) could be found in \cite{MR967638}, Theorem 2.2, where the author proved that if $\omega\geq\frac{p}{q}+$, then
 $$|P_{\frac{p}{q}+}(\xi)-P_\omega(\xi)|\leq 4C\theta(|\pi(\omega)q-p|),$$
where the constant $C=1200$. Similarly,  if $\omega\leq\frac{p}{q}-$, then
$$|P_{\frac{p}{q}-}(\xi)-P_\omega(\xi)|\leq 4C\theta(|\pi(\omega)q-p|),$$
where $C=1200$.

\begin{Cor}\label{continuous at irrational number}\textup{(\cite{MR920622})}
The map $\omega\mapsto P_{\omega,h}(\xi)$ is continuous at any irrational number $\omega$, uniformly in $\xi$.
\end{Cor}
\begin{Rem}
In general, $\omega\mapsto P_{\omega,h}$ is not continuous at rational symbol $\omega=p/q$.
\end{Rem}

\section{The conjunction operation}
The reason for not restricting our attention to $C^2$ function $h$ is that we consider not only  the generating functions with respect to $f\in\J$, but also the class of functions generated by the conjunction operation. Let $h_1$ and $h_2 :\R^2\rightarrow\R$ be two continuous functions satisfying the conditions $(\h1)-(\h5)$ and $(\h6\theta)$, \emph{ the conjunction} is defined in the following way:
$$h_1*h_2(x,x')=\min\limits_yh_1(x,y)+h_2(y,x').$$

Notice that even when both $h_1$ and $h_2$ are smooth,  the conjunction $h_1*h_2$ need not be smooth. However, we still have the following useful property.

\begin{Pro}\textup{(\cite{MR920622})}\label{conjunction}
If $h_1$ and $h_2 :\R^2\rightarrow\R$ are two continuous functions satisfying the conditions $(\h1)-(\h5)$ and $(\h6\theta)$, then $h_1*h_2$ satisfies $(\h1)-(\h5)$ and $(\h6\theta)$ with the same $\theta$.
\end{Pro}
Therefore,  one can still define the minimal configurations and Peierls barriers with respect to  $h_1*h_2$.

Given a function $h$ as descried above and a rational number $p/q$ (in lowest terms), we define the following conjunction
\begin{equation}\label{conjunct itself}
  H_{(q,p)}(x,x'):=h^{*q}(x,x'+p)=h*\cdot\cdot\cdot*h(x,x'+p),
\end{equation}
where $h^{*q}$ denotes the $q$-fold conjunction of $h$ with itself. It is easy to verify the following equivalent relations about the minimal configurations and the Peierls barrier functions.

\begin{Pro}\textup{(\cite{MR1323222}, \cite{MR967638})}\label{equivalence of conjunction}
If $h$  satisfies $(\h1)-(\h5)$ and $(\h6\theta)$, and $H_{(q,p)}(x,x')$ is the conjunction (\ref{conjunct itself}), then
$$\A_{\omega,h}=\A_{q\omega-p,H_{(q,p)}}~\text{and}~P_{\omega,h}(\xi)=P_{q\omega-p,H_{(q,p)}}(\xi),$$
where the rotation symbol $\omega\in(\R\setminus\Q)$ or $\omega=\frac{p_1}{q_1}\pm,\frac{p_1}{q_1}$ whose denominator $q_1$ is divisible by $q$.
\end{Pro}

\section{Main results}
\subsection{Monotone twist diffeomorphism}
Let $f$ be an exact area-preserving monotone twist diffeomorphism and let $\bar{f}$ be a lift of $f$ which satisfies $\bar{f}(x+1,y)=\bar{f}(x,y)+(1,0)$, and we set $\bar{f}(x,y)=(\bar{f_1}(x,y), \bar{f_2}(x,y))$. For the rotation interval $[\omega_0,\omega_1]$, there is a compact annulus $\an_K:=\s\times[-K,K]$ with sufficiently large $K=K(\omega_0,\omega_1,f)$ so that the minimal orbits of $f$ satisfies
\begin{equation}\label{basic setting}
 \m_{\omega,\bar{f}}\subseteq\an_{K-2}\subseteq\an_{K-1}\subseteq\an_K, ~\forall\omega\in[\omega_0,\omega_1].
\end{equation}

The space $C^1(\an_K)=C^1(\an_K,\R)$ is provided with the norm:
$$\|f\|_{C^1(\an_K)}=\sup\limits_{0\leq j\leq 1}\max\limits_{\an_K}|D^jf|.$$

\begin{The}\label{main theorem 1}
Let $f\in\J$, $[\omega_0,\omega_1]$ and $\an_K$  be as shown above. There exist positive numbers $\delta_0=\delta_0(\omega_0,\omega_1,f)\ll 1$ and $C_0=C_0(\omega_0,\omega_1,f)$ such that for any $f'\in\J$, if $\|f'-f\|_{C^1(\an_K)}\leq \delta_0$, then the corresponding Peierls barrier functions satisfy
$$|P_{\omega,f'}(\xi)-P_{\omega,f}(\xi)|\leq C_0\|f'-f\|_{C^1(\an_K)}^{\frac{1}{3}}~,~~\forall \xi\in\R$$
where $\omega\in(\R\setminus\Q)\bigcup(Q-)\bigcup{Q+}$, with $\pi(\omega)\in [\omega_0,\omega_1]$.
\end{The}
\begin{Rem}
\begin{enumerate}[(1)]
  \item The projection $\pi$ is defined in (\ref{underlying number}). The conclusion of Theorem \ref{main theorem 1} also holds for $C^r$ perturbations.
  \item The novelty here is that the constant $C_0$ does not depend on $\xi$ and the $1/3$-H\"{o}lder regularity is uniform in $\xi$, which could be derived from our proof below, see (\ref{final 3}).
  \item For rational symbol $\omega=p/q$, the conclusion of Theorem \ref{main theorem 1}  doesn't hold in general since the map $\omega\mapsto P_{\omega}$ is not continuous at rational symbol.
\end{enumerate}
\end{Rem}

In order to prove Theorem \ref{main theorem 1}, we need the following lemmas.

\begin{Lem}\label{lemma 1}
Under the same assumptions of Theorem \ref{main theorem 1}, let $h$, $h'$ be generating functions of $f$ and $f'$ respectively, satisfying the condition $h(0,0)=h'(0,0)=0$. There exist positive numbers $\delta_1=\delta_1(K,f)\ll 1$ and  $C_1=C_1(K,f)$, such that if $\|f'-f\|_{C^1(\an_K)}\leq \delta_1$, then
$$\|h'-h\|_{C^0(\B_{K-1})}\leq C_1\|f'-f\|_{C^1(\an_K)}$$
where $\B_{K-1}:=\{ (x,x')\in\R^2 | \bar{f_1}(x,-(K-1))\leq x'\leq\bar{f_1}(x,K-1) \}$.
\end{Lem}
\begin{proof}
Take $\delta_1$ small enough such that $\|f'-f\|_{C^1(\an_K)}\leq\delta_1\ll 1$, then the lifts
$$\|\bar{f'}-\bar{f}\|_{C^1([0,1]\times[-K,K])}=\|f'-f\|_{C^1(\an_K)}.$$
Since for all $k\in\Z$,  $\bar{f'}(x+k,y)=\bar{f'}(x,y)+(k,0)$ and $\bar{f}(x+k,y)=\bar{f}(x,y)+(k,0)$, one can deduce that
\begin{equation}\label{equivalence}
 \|\bar{f'}-\bar{f}\|_{C^1(\R\times[-K,K])}=\|\bar{f'}-\bar{f}\|_{C^1([0,1]\times[-K,K])}=\|f'-f\|_{C^1(\an_K)}.
\end{equation}

By adding a constant we can assume $h'(0,0)=h(0,0)=0$. By choosing suitably large $K$, we can assume $(0,0)\in \B_{K-2}$.
From the assumption (\ref{basic setting}), for all $\omega$ with $\pi(\omega)\in[\omega_0,\omega_1]$ and $x=(x_i)_{i\in\Z}\in\M_{\omega,h}$, we have $(x_i,x_{i+1})\in\B_{K-2}$, i.e.
$$\bar{f_1}(x_i,-(K-2))\leq x_{i+1}\leq\bar{f_1}(x_i,K-2).$$
Since $\bar{f'}$ and $\bar{f}$ are sufficiently close, then for  all $x'=(x'_i)_{i\in\Z}\in\M_{\omega,h'}$, we have $(x'_i,x'_{i+1})\in\B_{K-1}$, i.e.
$$\bar{f'_1}(x'_i,-(K-1))\leq x'_{i+1}\leq\bar{f'_1}(x'_i,K-1).$$
From now on, we only consider $h, h'$ restricted on the set $\B_{K-1}$ (See Figure 2).

In the following, we set the bounded region $\D:=\B_{K-1}\bigcap([0,1]\times\R)$.
\begin{figure}
  \centering
  \includegraphics[scale=0.6]{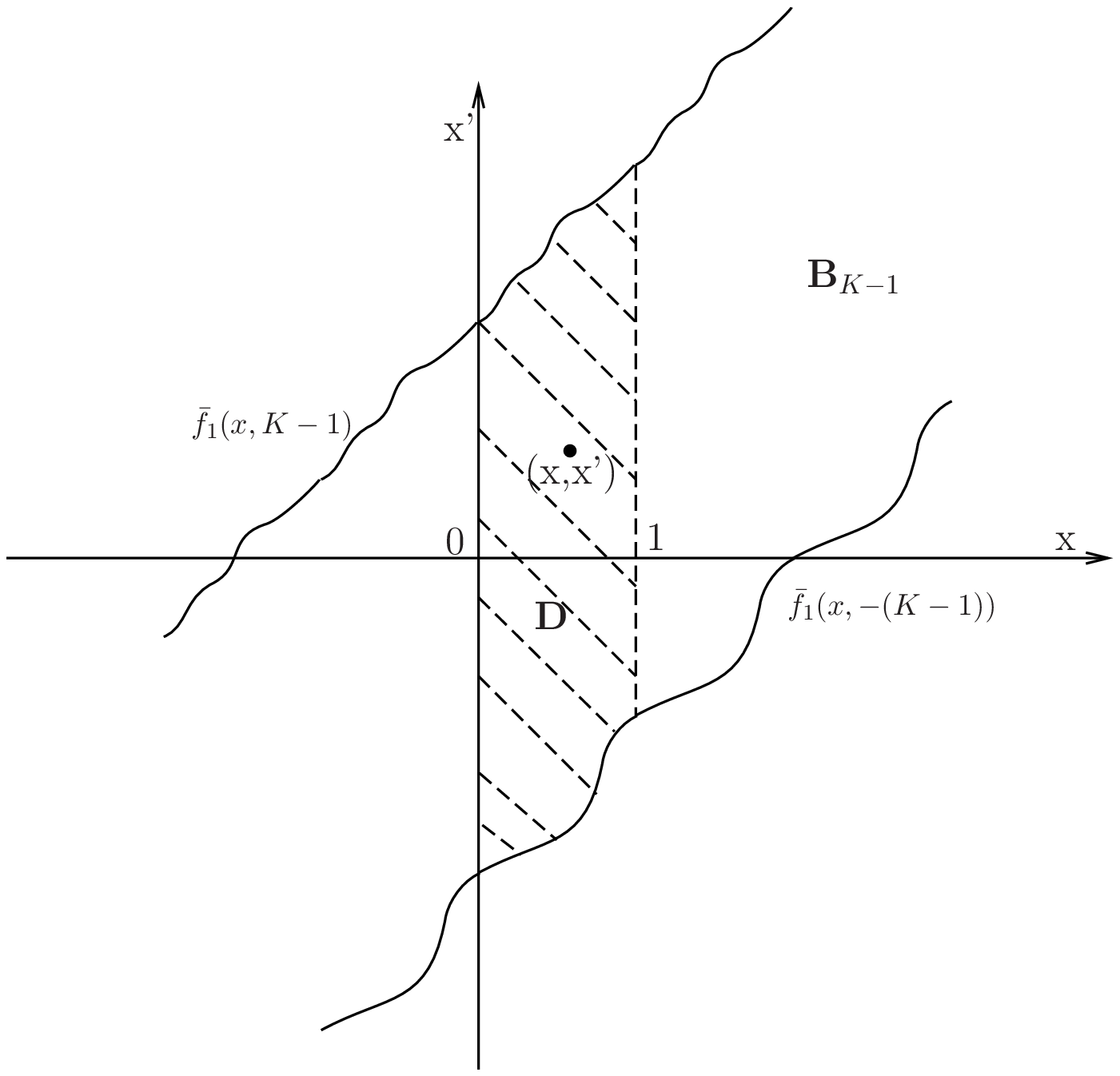}\\
  \caption{}\label{graph of domain}
\end{figure}
By the condition $(\h1)$, for all $(x,x')\in\B_{K-1}$,
\begin{equation}\label{results of lemma 1}
\|h'-h\|_{C^0(\B_{K-1})}=\|h'-h\|_{C^0(\D)}
\end{equation}

Because $h,h$ are $C^2$, so for $(x,x')\in\D$,
\begin{align*}
   \begin{split}
   h(x,x')=h(0,0)+\int_0^x\partial_1h(t,0)dt+\int_0^{x'}\partial_2h(x,t)dt  \\
   h'(x,x')=h'(0,0)+\int_0^x\partial_1h'(t,0)dt+\int_0^{x'}\partial_2h'(x,t)dt
   \end{split}
\end{align*}
and
\begin{align}\label{perturbation of h}
\begin{split}
|h'(x,x')-h(x,x')|&\leq\int_0^x|\partial_1h'(t,0)-\partial_1h(t,0)|dt+\int_0^{x'}|\partial_2h'(x,t)-\partial_2h(x,t)|dt\\
&\triangleq I_1+I_2
\end{split}
\end{align}
\paragraph{\textbf{Step 1}}
Firstly, let's estimate $I_1$. By (\ref{generating}), we have
\begin{equation}\label{generating of I_1}
\begin{cases}
0=\bar{f_1}(t,-\partial_1h(t,0)),\\
0=\bar{f'_1}(t,-\partial_1h'(t,0)),
\end{cases}
0\leq t\leq 1
\end{equation}
and by (\ref{generating of I_1}),
\begin{align}\label{I_1}
 \begin{split}
 0&=\bar{f_1}(t,-\partial_1h(t,0))-\bar{f'_1}(t,-\partial_1h'(t,0))\\
 &=\bar{f_1}(t,-\partial_1h(t,0))-\bar{f_1}(t,-\partial_1h'(t,0))+\bar{f_1}(t,-\partial_1h'(t,0))-\bar{f'_1}(t,-\partial_1h'(t,0))
 \end{split}
\end{align}

 If we set $a:=\min\limits_{(x,y)\in\R\times[-K,K]}\frac{\partial \bar{f_1}}{\partial y}(x,y)$, then $a>0$ since $\frac{\partial\bar{f_1}}{\partial y}>0$ and $\bar{f_1}(x+k,y)=\bar{f_1}(x,y)+(k,0).$

Because $\bar{f},\bar{f'}$ are sufficiently close, so it's easy to compute that
\begin{equation}\label{boundness of I_1}
|-\partial_1h(t,0))|\leq K-1,~|-\partial_1h'(t,0))|\leq K,~\forall 0\leq t\leq 1
\end{equation}
By (\ref{I_1}) and (\ref{boundness of I_1}),
\begin{align*}
|\bar{f_1}(t,-\partial_1h'(t,0))-\bar{f'_1}(t,-\partial_1h'(t,0))|&=|\bar{f_1}(t,-\partial_1h(t,0))-\bar{f_1}(t,-\partial_1h'(t,0))|\geq a|\partial_1h'(t,0)-\partial_1h(t,0)|.
\end{align*}
Therefore,
\begin{align*}
\|\bar{f'_1}-\bar{f_1}\|_{C^1(\R\times[-K,K])}\geq a|\partial_1h'(t,0)-\partial_1h(t,0)|
\end{align*}
and by (\ref{equivalence}),
\begin{align}\label{estimate of I_1}
I_1\leq\int_0^x\frac{1}{a}\|\bar{f'_1}-\bar{f_1}\|_{C^1(\R\times[-K,K])} dt\leq\int_0^x \frac{1}{a}\|f'-f\|_{C^1(\an_K)}dt\leq\frac{1}{a}\|f'-f\|_{C^1(\an_K)},
\end{align}
since $0\leq x\leq 1$.
\paragraph{\textbf{Step 2}}
Secondly, let's estimate $I_2$. For each fixed $x$, let $$\phi_x(y):=\bar{f_1}(x,y) \text{~and~} \phi'_x(y):=\bar{f'_1}(x,y).$$
Since $\frac{d\phi_x(y)}{dy}>0$, we could choose a positive number $0<b<1$ such that
\begin{equation}\label{phi}
  b\leq\frac{d\phi_x(y)}{dy}\leq\frac{1}{b} ~\textup{for all}~(x,y)\in[0,1]\times[-K,K].
\end{equation}
Because $\bar{f},\bar{f'}$ are sufficiently close, so it's easy to compute that for $(x,t)\in\D$,
\begin{equation}\label{boundness of I_2}
  |\phi^{-1}_x(t)|\leq K-1,~|(\phi'_x)^{-1}(t)|\leq K
\end{equation}
where $\phi^{-1}_x, (\phi'_x)^{-1}$ are the inverse function of $\phi_x, \phi'_x $ respectively.

Then, by (\ref{equivalence}), (\ref{phi}), (\ref{boundness of I_2}), we can conclude that for $(x,t)\in\D$,
\begin{align}\label{I_2}
\begin{split}
|(\phi'_x)^{-1}(t)-(\phi_x)^{-1}(t)|&=|(\phi_x)^{-1}\circ\phi_x\circ(\phi'_x)^{-1}(t)-(\phi_x)^{-1}(t)|\\
&\leq\max\limits_{(x,s)\in\D}|\frac{d(\phi_x)^{-1}(s)}{ds}|~|\phi_x\circ(\phi'_x)^{-1}(t)-t|\\
&\leq\frac{1}{b}|(\phi_x-\phi'_x+\phi'_x)\circ(\phi'_x)^{-1}(t)-t|\\
&=\frac{1}{b}|(\phi_x-\phi'_x)\circ(\phi'_x)^{-1}(t)|\\
&=\frac{1}{b}|\bar{f_1}(x,(\phi'_x)^{-1}(t))-\bar{f'_1}(x,(\phi'_x)^{-1}(t))|\\
&\leq\frac{1}{b}\|\bar{f'}-\bar{f}\|_{C^1([0,1]\times[-K,K])}=\frac{1}{b}\|f'-f\|_{C^1(\an_K)}.
\end{split}
\end{align}

By (\ref{generating}), we get
\begin{equation}\label{generating of I_2}
\begin{cases}
\partial_2h(x,t)=\bar{f_2}(x,\phi_x^{-1}(t)),\\
\partial_2h'(x,t)=\bar{f'_2}(x,(\phi'_x)^{-1}(t)),
\end{cases}
\end{equation}
and by (\ref{I_2}) and (\ref{generating of I_2}), one can deduce that
\begin{align}
\begin{split}
|\partial_2h'(x,t)-\partial_2h(x,t)|&\leq|\bar{f'_2}(x,(\phi'_x)^{-1}(t))-\bar{f_2}(x,(\phi'_x)^{-1}(t))|+|\bar{f_2}(x,(\phi'_x)^{-1}(t))-\bar{f_2}(x,\phi_x^{-1}(t))|\\
&\leq \|f'-f\|_{C^1(\an_K)}+\|\frac{\partial\bar{f_2}}{\partial y}\|_{C^0([0,1]\times[-K,K])}~|(\phi'_x)^{-1}(t)-(\phi_x)^{-1}(t)|\\
&\leq (1+\frac{L}{b})\|f'-f\|_{C^1(\an_K)}
\end{split}
\end{align}
where $L=\|\frac{\partial\bar{f_2}}{\partial y}\|_{C^0([0,1]\times[-K,K])}$.\\
Furthermore,
\begin{align}\label{estimates of I_2}
\begin{split}
I_2&\leq\int_0^{x'}(1+\frac{L}{b})\|f'-f\|_{C^1(\an_K)}dt\leq \max\limits_{x\in[0,1]}(~|\bar{f_1}(x,K)|+|\bar{f_1}(x,-K)|~) (1+\frac{L}{b})\|f'-f\|_{C^1(\an_K)}\\
&\leq 2\|\bar{f}\|_{C^1([0,1]\times[-K,K])}(1+\frac{L}{b})\|f'-f\|_{C^1(\an_K)}
\end{split}
\end{align}

Finally, by (\ref{results of lemma 1}), (\ref{perturbation of h}), (\ref{estimate of I_1}) and (\ref{estimates of I_2}), we get that
$$\|h'-h\|_{C^0(\B_{K-1})}\leq C_1\|f'-f\|_{C^1(\an_K)}$$
with the constant $C_1$ only depending on $K$ and $f$.
\end{proof}

The following lemma give an equivalent definition of the Peierls barrier for the rational rotation symbol.
\begin{Lem}\label{lemma 3}
Let $h$ be a continuous real valued function satisfying $(\h1)-(\h5)$ and $(\h6\theta)$, then for rational rotation symbol $\frac{p}{q}\in\Q$(in lowest terms), the Peierls barrier has an equivalent defintion:
$$P_{\frac{p}{q},h}(\xi)=\min_{\substack{y_0=\xi\\y_q=y_0+p}}\sum\limits_{i=0}^{q-1}h(y_i,y_{i+1})-\min_{\substack{x_q=x_0+p}}\sum\limits_{i=0}^{q-1}h(x_i,x_{i+1}).$$
\end{Lem}
\begin{proof}
Take $\xi\in(J_-,J_+)$, where $(J_-,J_+)$ is a complementary interval to $\A_{\frac{p}{q},h}$. Let $x^\pm=(x^\pm_i)_{i\in\Z}$ satisfying $x^\pm_{i+q}=x_i^\pm+p, x^\pm_0=J_{\pm}$ be the periodic minimal configurations in $\M_{\frac{p}{q},h}$.

Comparing it with the definition (\ref{the definiton of barrier}), we only need to prove that the minimal segment $(\xi=y_0,y_1,...,y_{q-1},$ $y_q=y_0+p)$ which achieves the minimum in the definition (\ref{the definiton of barrier}) satisfies the constraint
$$x^-_i\leq y_i\leq x^+_i.$$
In fact, we claim that the Aubry graphs of $x^-, y, x^+$ do not cross. Since $x^\pm$ are minimal configurations, $(y_i)_{i=0}^q$ is a minimal segment, the claim is an easy consequence according to Aubry's crossing lemma (Lemma \ref{Aubry crossing lemma}).
\end{proof}

\begin{Lem}\label{lemma 4}
Let $h$ be a continuous real valued function satisfying $(\h1)-(\h5)$ and $(\h6\theta)$, then the Peierls barriers have the following properties:
\begin{enumerate}[(1)]
  \item If $\frac{1}{q}>0$, then
  $$|P_{\frac{1}{q}}(\xi)-P_{0+}(\xi)|\leq\frac{16\theta}{q}.$$
  \item If $\frac{1}{q}>0$, then
  $$|P_{-\frac{1}{q}}(\xi)-P_{0-}(\xi)|\leq\frac{16\theta}{q}.$$
\end{enumerate}
\end{Lem}
\begin{proof}
Notice that this lemma is not a direct conclusion of Proposition \ref{continuous of modulus}. Lemma \ref{lemma 4} (1) was firstly claimed and proved by Mather in (\cite{MR967638}, (4.1), (4.4a), (4.4b)), Lemma \ref{lemma 4} (2) can be proved by the same approach.
\end{proof}
For symbol simplicity, we denote by $\E=\{(x,x')\in\R^2|~|x'-x|\leq 5\}.$
\begin{Lem}\label{lemma 5}
Assume that $h, h'$ are the generating functions described in Lemma \ref{lemma 1}. Given $\frac{p}{q}\in[\omega_0,\omega_1]$, $q>0$ and $p,q$ are relatively prime,  and let $H_{(q,p)}$ (resp. $H'_{(q,p)}$) be the conjunction (\ref{conjunct itself}) of $h$ (resp. $h'$), we have the following estimates:
\begin{enumerate}[(1)]
\item $\|H'_{(q,p)}-H_{(q,p)}\|_{C^0(\E)}\leq q\|h'-h\|_{C^0(\B_{K-1})}.$
\item For $0<m\in\Z$,
$$|P_{\frac{1}{m},H'_{(q,p)}}(\xi)-P_{\frac{1}{m},H_{(q,p)}}(\xi)|\leq 2mq\|h'-h\|_{C^0(\B_{K-1})}. $$

\end{enumerate}
\end{Lem}
\begin{proof}
(1). In fact, take $(x,x')\in\E$ and let $(x=x_0,x_1,...,x_{q-1},x_q=x'+p)$ be the minimal segment which achieves the minimum of $H_{(q,p)}(x,x')$, i.e.
$$H_{(q,p)}(x,x')=h(x_0,x_1)+\cdot\cdot\cdot+h(x_{q-1},x_q).$$
Then,
\begin{align}\label{H}
 H'_{(q,p)}(x,x')\leq\sum\limits_{i=0}^{q-1}h'(x_i,x_{i+1})=H_{(q,p)}(x,x')+\sum\limits_{i=0}^{q-1}(h'(x_i,x_{i+1})-h(x_i,x_{i+1})).
\end{align}

Next, we need to show that
\begin{equation}\label{HH}
 (x_i,x_{i+1})\in\B_{K-1},\quad\forall 0\leq i\leq q-1
\end{equation}

Recall that, by Proposition \ref{bangert}, there exists $\alpha\in\R$ such that $|x_{q}-x_0-q\alpha|<2.$
Then, we obtain
$$|\alpha|<\frac{2+p+|x'-x|}{q}\leq \frac{7}{q}+\frac{p}{q},$$
and by Proposition \ref{bangert} again,
$$|x_{i+1}-x_i|\leq 2+|\alpha|\leq9+\frac{p}{q}\leq10+\max\{|\omega_0|,|\omega_1|\}.$$
Notice that we have assumed that the constant $K$ is large enough, thus (\ref{HH}) holds.

By (\ref{H}) and (\ref{HH}), we deduce that
$$H'_{(q,p)}(x,x')\leq H_{(q,p)}(x,x')+q\|h'-h\|_{C^0(\B_{K-1})}, ~\forall(x,x')\in\E.$$
Similarly, we can prove that
$$H_{(q,p)}(x,x')\leq H'_{(q,p)}(x,x')+q\|h'-h\|_{C^0(\B_{K-1})}, ~\forall(x,x')\in\E.$$
This completes the proof of (1).

(2). Firstly, we claim that the definition of  $P_{\frac{1}{m},H_{(q,p)}}$ only depends on $H_{(q,p)}(x,x')$ restricted on $\E$.

Take $\xi\in(J_-,J_+)$, where $(J_-,J_+)$ is a complementary interval to $\A_{\frac{1}{m},h}$. Let $x^\pm=(x^\pm_i)_{i\in\Z}$ satisfying $ x^\pm_0=J_{\pm}$ be the $\frac{1}{m}$-minimal configurations of $\M_{\frac{1}{m},h}$,
$$P_{\frac{1}{m},H_{(q,p)}}(\xi)=\min\{\sum\limits_{i=0}^{m-1}H_{(q,p)}(y_i,y_{i+1})-H_{(q,p)}(x^-_i,x^-_{i+1}) | y_0=\xi\}$$
where the minimum is taken over the set of all configurations satisfying $x_i^-\leq y_i\leq x_i^+$ and $y_{i+m}=y_i+1$. Since $x_i^-\leq y_i\leq x_i^+, 0\leq i\leq m-1$, by the totally ordered property of $x^\pm$, we get
\begin{align}
\begin{split}
|y_{i+1}-y_i|&\leq|x^+_{i+1}-x^-_i|\leq|x^+_{i+1}-x^-_{i+1}|+|x^-_{i+1}-x^-_i|\leq 1+|x^-_{i+1}-x^-_i|\\
&\leq 2+\pi(\omega) \quad\quad(\textup{By Proposition \ref{rotation estimates}})\\
&\leq 3
\end{split}
\end{align}
Thus,
$$(y_i,y_{i+1})\in\E,\quad\forall 0\leq i\leq m-1,$$
which proves our claim. Similarly, we can prove this for $P_{\frac{1}{m},H'_{(q,p)}}$.

On the other hand, by Lemma \ref{lemma 3}, we have
\begin{align*}
P_{\frac{1}{m},H_{(q,p)}}(\xi)=\min_{\substack{y_0=\xi\\y_m=y_0+1}}\sum\limits_{i=0}^{m-1}H_{(q,p)}(y_i,y_{i+1})-\min_{\substack{x_m=x_0+1}}\sum\limits_{i=0}^{m-1}H_{(q,p)}(x_i,x_{i+1}) \\
P_{\frac{1}{m},H'_{(q,p)}}(\xi)=\min_{\substack{y_0=\xi\\y_m=y_0+1}}\sum\limits_{i=0}^{m-1}H'_{(q,p)}(y_i,y_{i+1})-\min_{\substack{x_m=x_0+1}}\sum\limits_{i=0}^{m-1}H'_{(q,p)}(x_i,x_{i+1})  \end{align*}
It follows that
\begin{align}\label{lemma 5(2)}
|P_{\frac{1}{m},H'_{(q,p)}}(\xi)-P_{\frac{1}{m},H_{(q,p)}}(\xi)|\leq |I_1|+|I_2|.
\end{align}
where $$I_1=\min_{\substack{y_0=\xi\\y_m=y_0+1}}\sum\limits_{i=0}^{m-1}H'_{(q,p)}(y_i,y_{i+1})-\min_{\substack{y_0=\xi\\y_m=y_0+1}}\sum\limits_{i=0}^{m-1}H_{(q,p)}(y_i,y_{i+1})$$ and
$$I_2=\min_{\substack{x_m=x_0+1}}\sum\limits_{i=0}^{m-1}H'_{(q,p)}(x_i,x_{i+1})-\min_{\substack{x_m=x_0+1}}\sum\limits_{i=0}^{m-1}H_{(q,p)}(x_i,x_{i+1})$$

Firstly, let's estimate $|I_1|$. Assume that $(\xi=a_0,a_1,...,a_m=a_0+1)$ is the minimal segment satisfying
$$\sum\limits_{i=0}^{m-1}H_{(q,p)}(a_i,a_{i+1})=\min_{\substack{y_0=\xi\\y_m=y_0+1}}\sum\limits_{i=0}^{m-1}H_{(q,p)}(y_i,y_{i+1}).$$
We derive from Proposition \ref{bangert} that there exists $\alpha\in\R$ such that
$$|a_m-a_0-m\alpha|\leq 2$$
and since $a_m=a_0+1$, we obtain
$$|\alpha|\leq\frac{3}{m}\leq 3  \textup{~and~} |a_{i+1}-a_{i}|\leq 2+|\alpha|\leq 5 .$$
Consequently, $(a_i,a_{i+1})\in\E,\quad\forall~ 0\leq i\leq m-1.$ Then,
\begin{align}\label{part 1}
\begin{split}
\min_{\substack{y_0=\xi\\y_m=y_0+1}}\sum\limits_{i=0}^{m-1}H'_{(q,p)}(y_i,y_{i+1})&\leq\sum\limits_{i=0}^{m-1}H'_{(q,p)}(a_i,a_{i+1})\leq\sum\limits_{i=0}^{m-1}H_{(q,p)}(a_i,a_{i+1})\\
&+\sum\limits_{i=0}^{m-1}(H'_{(q,p)}(a_i,a_{i+1})-H_{(q,p)}(a_i,a_{i+1}))\\
&\leq \min_{\substack{y_0=\xi\\y_m=y_0+1}}\sum\limits_{i=0}^{m-1}H_{(q,p)}(y_i,y_{i+1})+m\kappa,
\end{split}
\end{align}
where $\kappa=\|H'_{(q,p)}-H_{(q,p)}\|_{C^0(\E)}$. Similarly, we can prove
\begin{align}\label{part 2}
\min_{\substack{y_0=\xi\\y_m=y_0+1}}\sum\limits_{i=0}^{m-1}H_{(q,p)}(y_i,y_{i+1})\leq\min_{\substack{y_0=\xi\\y_m=y_0+1}}\sum\limits_{i=0}^{m-1}H'_{(q,p)}(y_i,y_{i+1})+m\kappa.
\end{align}
We conclude that, by (\ref{part 1}) and (\ref{part 2}),
\begin{equation}\label{I1111}
  |I_1|\leq m\|H'_{(q,p)}-H_{(q,p)}\|_{C^0(\E)}
\end{equation}

Secondly, let's estimate $|I_2|$. In fact, it can be similarly estimated as $|I_1|$,
\begin{equation}\label{I2222}
|I_2|\leq m\|H'_{(q,p)}-H_{(q,p)}\|_{C^0(\E)}
\end{equation}
Therefore, it follows from (\ref{I1111}), (\ref{I2222}) and Lemma \ref{lemma 5} (1) that
$$(\ref{lemma 5(2)})\leq 2m\|H'_{(q,p)}-H_{(q,p)}\|_{C^0(\E)}\leq 2mq\|h'-h\|_{C^0(\B_{K-1})},$$
which completes the proof of (2).
\end{proof}

Because our proof of Theorem \ref{main theorem 1} needs the technique of rational approximation, so we introduce the following lemma.
\begin{Lem}[Dirichlet approximation]\label{lemma 2}
Given $\omega\in\R$ and $0<n\in\Z$, then there exists a rational number $\frac{p}{q}$, $q>0$ and $p,q$ are relatively prime, such that
$$0<q\leq n,~|\omega-\frac{p}{q}|\leq\frac{1}{q(n+1)}.$$
\end{Lem}
\begin{proof}
It can be easily proved by pigeon hole principle. Indeed, we firstly assume that $\omega\in\R\setminus\Q$. For every $0<k\leq n$, we can find a integer $h_k$ such that
$$a_k:=k\omega+h_k\in(0,1).$$
Then $a_1,...,a_n$ are $n$ distinct points in the interval $(0,1)$ since $\omega\in\R\setminus\Q$. By pigeon hole principle,
there exist two points $a_i,a_j ( i<j)$ such that $|a_i-a_j|\leq\frac{1}{n+1}$, i.e.
$$|(i-j)\omega-(h_j-h_i)|\leq\frac{1}{n+1}$$
Thus, $$|\omega-\frac{h_j-h_i}{i-j}|\leq\frac{1}{(j-i)(n+1)},$$
which completes the proof for all $\omega\in\R\setminus\Q$.

For $\omega\in\Q$, it can be proved similarly.
\end{proof}
\begin{proof}[Proof of Theorem \ref{main theorem 1}]
To simplify notations, let's set
$$\delta:=\|f'-f\|_{C^1({\an_K})}$$
Because we only concern what happens in the compact region $\an_K$, so we can assume that the generating function $h$ of $f$
satisfies $(\h1)-(\h5)$ and $(\h6\theta)$, with some $\theta=\theta(K,f)$ depending only on $K$ and $f$.

In view of the geometrical meaning of $\theta$, we know that, restricted on the compact region $\an_K$, $f|_{\an_K}$ turns every vertical vector to the right by an angle by at least $\beta$ ($\cot\beta=\theta, 0<\beta<\pi/2$). Thus there exists $0<\delta_2=\delta_2(K,f)$ such that for the perturbation $f'$ of $f$, if $\delta\leq\delta_2$, the diffeomorphism $f'|_{\an_K}$ turns every vertical vector to the right by an angle by at least $\beta'$ where $\beta'<\beta$ and $\cot\beta'=2\theta$.

From now on, we set $\delta_0=\min\{\delta_1,\delta_2\}\ll 1$. So if $\delta\leq\delta_0$, the generating function $h'$ of $f'$ satisfying $(\h1)-(\h5)$ and $(\h6\theta')$ with
$$\theta'=2\theta.$$
Meanwhile, by Lemma \ref{lemma 1}, we choose the generating functions such that
$h(0,0)=h'(0,0)=0$ and
\begin{equation}\label{what}
 \|h'-h\|_{C^0(\B_{K-1})}\leq C_1\delta.
\end{equation}

According to the proof of Lemma \ref{lemma 1}, for all rotation symbol $$\omega\in(\R\setminus\Q)\bigcup(\Q+)\bigcup(\Q-)$$ with $\pi(\omega)\in[\omega_0,\omega_1]$, the minimal configurations $\M_{\omega,h}, \M_{\omega,h'}\subseteq\B_{K-1}.$  Thus the definition of Peierls barrier functions $P_{\omega,h}, P_{\omega,h'}$ only depend on  $h|_{B_{K-1}}, h'|_{B_{K-1}}.$

Let's begin our proof. Firstly, we approximate $\omega$ by rational number. In fact, by Lemma \ref{lemma 2}, for
$n=[ \delta^{-\frac{1}{3}} ]$, we could find a rational number $\frac{p}{q}$ (in lowest terms) such that
\begin{equation}\label{approximation}
 0<q\leq n\leq\delta^{-\frac{1}{3}}+1,  \quad|q\pi(\omega)-p|\leq\frac{1}{n+1}\leq\delta^{\frac{1}{3}}.
\end{equation}

Assume that $\omega\geq\frac{p}{q}+$  (the case $\omega\leq\frac{p}{q}-$ is similar). By Proposition \ref{continuous of modulus}, Proposition \ref{equivalence of conjunction} and the estimate (\ref{approximation}), we have
\begin{align}\label{final 1}
\begin{split}
\|P_{\omega,h'}-P_{\omega,h}\|&\leq\|P_{\omega,h'}-P_{\frac{p}{q}+,h'}\|+\|P_{\frac{p}{q}+,h'}-P_{\frac{p}{q}+,h}\|+\|P_{\frac{p}{q}+,h}-P_{\omega,h}\|\\
&\leq 4800\theta'|q\pi(\omega)-p|+\|P_{0+,H'_{(q,p)}}-P_{0+,H_{(q,p)}}\|+4800\theta|q\pi(\omega)-p| \\
&\leq 14400\theta\delta^{\frac{1}{3}}+\|P_{0+,H'_{(q,p)}}-P_{0+,H_{(q,p)}}\|
\end{split}
\end{align}

Next, we only need to give the estimate of $\|P_{0+,H'_{(q,p)}}-P_{0+,H_{(q,p)}}\|$. In fact,  we take $m=[\delta^{-\frac{1}{3}}]$, and by Lemma \ref{lemma 4}, Lemma \ref{lemma 5} (2), we obtain
\begin{align*}
\begin{split}
\|P_{0+,H'_{(q,p)}}-P_{0+,H_{(q,p)}}\|&\leq\|P_{0+,H'_{(q,p)}}-P_{\frac{1}{m},H'_{(q,p)}}\|+\|P_{\frac{1}{m},H'_{(q,p)}}-P_{\frac{1}{m},H_{(q,p)}}\|+\|P_{\frac{1}{m},H_{(q,p)}}-P_{0+,H_{(q,p)}}\|\\
&\leq\frac{16}{m}\theta'+2qm\|h'-h\|_{C^0(\B_{K-1})}+\frac{16}{m}\theta\\
&\leq \frac{48}{m}\theta+2qmC_1\delta \quad (~\textup{By}~ (\ref{what})~)\\
\end{split}
\end{align*}
We derive from (\ref{approximation}) and $m=[\delta^{-\frac{1}{3}}]$ that
\begin{equation}\label{final 2}
\|P_{0+,H'_{(q,p)}}-P_{0+,H_{(q,p)}}\|\leq49\theta\delta^{\frac{1}{3}}+3C_1\delta^{\frac{1}{3}}.
\end{equation}
Finally, combining (\ref{final 1}) with (\ref{final 2}),
\begin{equation}\label{final 3}
\|P_{\omega,h'}(\xi)-P_{\omega,h}(\xi)\|\leq(14449\theta+3C_1)\delta^{\frac{1}{3}}=C_0\|f'-f\|_{C^1({\an_K})}^\frac{1}{3}
\end{equation}
where $C_0=14449\theta+3C_1$ which  only depends on $K, f$. We know that $K$ depends on $\omega_0,\omega_1,f$, so $C_0=C_0(\omega_0,\omega_1,f)$. This completes the proof.
\end{proof}

\subsection{Lagrangians with one and a half degrees of freedom}
In \cite{MR863203}, Moser showed that any monotone twist diffeomorphism on the cylinder $\s\times\R$ can be regarded as the time-1 map of a periodic Lagrangian system. In what follows, we specialize to the Lagrangians with one and a half degrees of freedom. Firstly, let's briefly recall some basic notions and results of Mather theory. For proofs and details, we refer to \cite{MR1109661}, \cite{MR1275203}.

Let $L:T\s\times\s\rightarrow \R$ be a $C^r (r\geq2)$ Lagrangian
satisfying Tonelli conditions:
\begin{enumerate}[(L1)]
  \item Convexity: For each $(x,t)\in\s\times\s$,  $L$ is strictly convex in $v$ coordinate.
  \item Superlinearity: $$\lim\limits_{\|v\|\rightarrow +\infty}\frac{L(x,v,t)}{\|v\|}=+\infty,\quad \text{uniformly on~} (x,t).$$
  \item Completeness: All solutions of the Euler-Lagrange equation
  $$\frac{d}{dt}(\frac{\partial L}{\partial v}(x,\dot{x},t))=\frac{\partial L}{\partial x}(x,\dot{x},t)$$
   are well defined for the whole $t\in\R$.
\end{enumerate}

Let $I=[a,b]$ be an interval, and $\gamma:I\rightarrow \s$ be an absolutely continuous curve, we denote by $$A(\gamma):=\int_a^b L(d\gamma(t),t) dt$$ the action of $\gamma$. An absolutely curve $\gamma: I\rightarrow \s$ is called a\emph{ minimizer} or \emph{action minimizing curve} if
$$A(\gamma)=\min_{\substack{\xi(a)=\gamma(a),\xi(b)=\gamma(b)\\ \xi\in C^{ac}(I,\s)}}\int_a^b L(d\xi(t),t) dt.$$
We call $\gamma:(-\infty,+\infty)\rightarrow \s$ a \emph{globally minimizing curve} if for all $a<b, \gamma$ is a minimizer on $[a,b]$. Notice that the minimizer satisfies the Euler-Lagrange equation.

Let $\mathcal{M}_L$ be the space of Euler-Lagrangian flow invariant probability measures on $T\s\times\s$. To each $\mu\in\mathcal{M}_L$, note that $\int \lambda d\mu$=0 for each exact 1-form $\lambda$. Therefore, given $c\in H^1(\s,\R)$ and a closed 1-form $\eta_c\in c=[\eta_c]$, we can define\emph{ Mather's $\alpha$ function}
$$\alpha(c):=-\inf\limits_{\mu\in\mathcal{M}_L}A_c(\mu)=-\inf\limits_{\mu\in\mathcal{M}_L}\int_{TM\times\s} L-\eta_c d\mu.$$
It's easy to be checked that $\alpha(c)$ is finite everywhere, convex and superlinear.

We associate to $\mu\in\mathcal{M}_L$ its rotation vector $\rho(\mu)\in H_1(\s,\R)$ in the following sense:
$$\langle\rho(\mu),[\eta_c]\rangle=\int_{TM\times\s}\eta_c d\mu,\quad \forall c\in H^1(\s,\R).$$
So we can define \emph{Mather's $\beta$ function}:
$$\beta(\omega):=\inf\limits_{\mu\in\mathcal{M}_L, \rho(\mu)=\omega}\int Ld\mu,\quad \forall \omega\in H_1(\s,\R).$$
$\beta$ is finite, convex, superlinear and $\beta$ is the Legendre-Fenchel dual of $\alpha$ .

\begin{Pro}\textup{(\cite{MR1139556})}\label{alpha and beta}
If $L: T\s\times\s\rightarrow\R$ is a Tonelli Lagrangian, then
\begin{enumerate}[(1)]
  \item the function $\beta:H_1(\s,\R)\equiv\R\rightarrow\R$ is strictly convex and differentiable at all $\omega\in\R\setminus\Q.$
  \item the function $\alpha: H^1(\s,\R)\equiv\R\rightarrow\R$ is differentiable everywhere.
\end{enumerate}
\end{Pro}
Next, we introduce the generalization of Peierls barrier to several degrees of freedom. For each $c\in H^1(\s,\R)$ and $n\in\Z^+$, we define a function $A^n_c$,
$$A^n_c:\s\times \s\rightarrow \R$$
$$A^n_c(x,x'):= \inf_{\substack{\gamma(0)=x, \gamma(n)=x'\\ \gamma\in C^{ac}([0,n],\s)}}\int_0^n (L-\eta_c)(d\gamma(s),s)ds .$$
Then, following Mather, we introduce the barrier function on $\s\times \s$
\begin{equation}\label{barrier h}
h_c^\infty(x,x'):=\liminf\limits_{n\rightarrow+\infty}A_c^n(x,x')+n\alpha(c).
\end{equation}
This function is useful in the construction of connecting orbits (see \cite{MR1275203}).

By Proposition \ref{alpha and beta}, $\alpha'(c)$ exists for every $c\in H^1(\s,\R)$, and the flat piece of graph $\alpha$ has rational slope.
\begin{Pro}\label{barrier equivalence}\textup{(\cite{MR1275203}, Proposition 7.1 and 7.2)}
Let $L$ be a Tonelli Lagrangian whose time-1 map is an area-preserving twist diffeomorphism, then
\begin{enumerate}[(1)]
  \item For $\omega\in\R\setminus\Q$, there exists a unique $c=c(\omega)$ such that $\alpha'(c)=\omega$, and
  $$h_c^\infty(x,x)=P_\omega(x).$$
  \item For rational number $\frac{p}{q}$ (in lowest terms), let $c_+:=\max\{c: \alpha'(c)=\frac{p}{q}\}$, $c_-:=\min\{c: \alpha'(c)=\frac{p}{q}\}$, then
      $$h_{c_+}^\infty(x,x)=P_{\frac{p}{q}+}(x),\quad h_{c_-}^\infty(x,x)=P_{\frac{p}{q}-}(x)$$
\end{enumerate}
\end{Pro}
\begin{Rem}
$[c_-,c_+]$ corresponds to the flat of graph $\alpha$ with rational slope $\frac{p}{q}$.
\end{Rem}
\begin{itemize}[$\bullet$]
\item Now, let $L:T\s\times\s\rightarrow\R$ be a Tonelli Lagrangian whose time-1 map $\Phi_1$ is an exact area-preserving monotone twist diffeomorphism. For the rotation interval $[\omega_0,\omega_1]$, there is a compact annulus $\an_K:=\s\times[-K,K]$ with sufficiently large $K=K(\omega_0,\omega_1,L)$ so that the minimal orbits of $\Phi_1$ satisfies
\begin{equation*}
 \m_{\omega,\Phi}\subseteq\an_{K-2}\subseteq\an_{K-1}\subseteq\an_K, ~\forall\omega\in[\omega_0,\omega_1].
\end{equation*}

The space $C^2(\an_{2K})=C^2(\an_{2K},\R)$ is provided with the norm:
$$\|f\|_{C^2(\an_{2K})}=\sup\limits_{0\leq j\leq 2}\max\limits_{\an_{2K}}|D^jf|.$$

\item $H^1(\s,\R)\equiv\R$, by abuse of notation, we use the same symbol $c$ to denote the real number in $\R$ or the closed 1-form $c d\vartheta$ of $\s$.
\item Let $h^\infty_{L_c}$ denote the barrier function (\ref{barrier h}) associated to the Lagrangian $L_c:=L-c$.

\item By Proposition \ref{alpha and beta} and \ref{barrier equivalence}, for irrational number $\omega\in H_1(\s,\R)\equiv\R$, we denote by $c(\omega)\in H^1(\s,\R)\equiv\R$ the unique number satisfying $\alpha'(c)=\omega. $ In addition, we can define $c_+(\omega), c_-(\omega)$ for rational numbers as Proposition \ref{barrier equivalence} (2). Similarly, let the Lagrangian $L'$ be a perturbation of $L$, one can also define $c'(\omega), c'_+(\omega), c'_-(\omega)$ in the same way.
\end{itemize}

\begin{The}\label{main theorem 2}
Let $L, ~[\omega_0,\omega_1],~\an_{2K}$ be as shown above. There exist constants $\delta_0=\delta_0(\omega_0,\omega_1,L)\ll1$ and $C_0=C_0(\omega_0,\omega_1,L)$ such that for any Lagrangian $L'$, if $\|L'-L\|_{C^2(\an_{2K}\times\s)}\leq\delta_0$ and  $\omega\in[\omega_0,\omega_1]$, we have,
\begin{enumerate}[(1)]
  \item for $\omega\in\R\setminus\Q$, $|h^\infty_{L'_{c'(\omega)}}(x,x)-h^\infty_{L_{c(\omega)}}(x,x)|\leq C\|L'-L\|^{\frac{1}{3}}_{C^2(\an_{2K}\times\s)}$
  \item for rational number $\omega$,  $|h^\infty_{L'_{{c'}_+(\omega)}}(x,x)-h^\infty_{L_{c_+(\omega)}}(x,x)|\leq C\|L'-L\|^{\frac{1}{3}}_{C^2(\an_{2K}\times\s)}$ and
      $$|h^\infty_{L'_{{c'}_-(\omega)}}(x,x)-h^\infty_{L_{c_-(\omega)}}(x,x)|\leq C\|L'-L\|^{\frac{1}{3}}_{C^2(\an_{2K}\times\s)}$$
\end{enumerate}
\end{The}

\begin{Rem} For the Lagrangian of many degrees of freedom, the barrier function $h_c^\infty(x,x')$ could also be defined, one may ask whether there are similar results in this case. In general, it is not true and there're counterexamples which could be found in \cite{MR3565373}. However, under additional assumptions, such as nearly-integrable Lagrangians of arbitrary degrees of freedom, we can also obtain some similar results, see \cite{MR3565373} for details.
\end{Rem}
\begin{proof}[Proof of Theorem \ref{main theorem 2}]
The time-1 map $\Phi_1$ of the Lagrangian $L$ is an exact area-preserving monotone twist diffeomorphism. Then, there exists $\delta_0=\delta_0(\omega_0,\omega_1,L)$ such that if $\|L'-L\|_{C^2(\an_{2K}\times\s)}\leq\delta_0$, the time-1 map $\Phi'_1$ of $L'$ is also a monotone twist diffeomorphism satisfying
\begin{equation}\label{22222}
\|\Phi'_1-\Phi_1\|_{C^1({\an_K})}\leq D\|L'-L\|_{C^2({\an_{2K}\times\s})}
\end{equation}
with the constant $D$. Therefore, by Proposition \ref{barrier equivalence}, it's not hard to observe that Theorem \ref{main theorem 2} is an easy consequence of (\ref{22222}) and Theorem \ref{main theorem 1}.
\end{proof}

\section{Application}
The destruction of invariant circles or Lagrangian tori (Converse KAM theory) is an interesting and important topic in dynamic systems (see for example, \cite{MR3061774} \cite{MR1279471} \cite{MR728564} \cite{MR967638} ). In this section, by applying Theorem \ref{main theorem 1}, we give an open and dense property about the destruction of invariant circles. A real number $\omega\in\R\setminus\Q$ is \emph{Diophantine} if there exist constants $C>0$ and $\tau> 1$ such that
$$|q\omega-p|\geq \frac{C}{|q|^\tau}\quad \textup{~for all~} p, q\in\Z, q\neq0.$$
A real number is \emph{Liouville} if it's not Diophantine .
\begin{The}\label{application}
Let $\J^r (r\geq 1)$ be the set of all $C^r$ exact area-preserving  monotone twist diffeomorphisms and let $\omega$ be a Liouville number. Then there exists a set $\mathcal{O}$ which is open and dense in $\J^r$ in the $C^r$ topology, such that for all $f\in\mathcal{O}$, there is no homotopically non-trivial $f-$invariant circle with rotation number $\omega$.
\end{The}
\begin{proof}
\textbf{Case $\textup{\uppercase\expandafter{\romannumeral1}}$}: $\omega$ is irrational.

Denote by $\mathcal{O}$ the set of all $C^r$ exact area-preserving  monotone twist diffeomorphisms that don't admit any homotopically non-trivial invariant circles with rotation number $\omega$. We only need to prove that the set $\mathcal{O}$ is open and dense in $\J^r$.

Given an exact area-preserving  monotone twist diffeomorphism $f\in J^r$, we assume that $f$ admits a homotopically non-trivial $f-$invariant circle with rotation number $\omega$. Then by Theorem 2.1 in \cite{MR967638}, for any neighbourhood $\mathcal{U}_f$ of $f$ in $\J^r$, we could find $g\in \mathcal{U}_f$ which does not admit any  homotopically non-trivial $g-$invariant circle of rotation number $\omega$. So $g\in\mathcal{O}\bigcap\mathcal{U}_f$, which proves that $\mathcal{O}$ is a dense set in $\J^r$.

On the other hand, we know that the Peierls barrier $P_{\omega,f}(\xi)\equiv0$ if and only if there exists a homotopically non-trivial $f-$invariant circle of rotation number $\omega$ (see \cite{MR920622}). Take $f\in\mathcal{O}$, then there exists a point $\xi_0\in\R$ such that $P_{\omega,f}(\xi_0)=a_0>0$. By Theorem \ref{main theorem 1}, we have
$$|P_{\omega,f'}(\xi)-P_{\omega,f}(\xi)|\leq C_0\|f'-f\|_{C^1}^{\frac{1}{3}}\leq C_0\|f'-f\|_{C^r}^{\frac{1}{3}}.$$
Thus, there exists a small neighbourhood $\mathcal{V}_f$ of $f$ in $\J^r$ such that for all $f'\in\mathcal{V}_f$, we have $$P_{\omega,f'}(\xi_0)\geq a_0/2>0.$$  Then $\mathcal{V}_f\subseteq\mathcal{O}$, which proves that $\mathcal{O}$ is an open set.

\textbf{Case $\textup{\uppercase\expandafter{\romannumeral 2}}$}: $\omega={\frac{p}{q}}$ is rational.

There is no homotopically non-trivial $f-$invariant circle of rotation number $\frac{p}{q}$ if and only if the Peierls barrier
$P_{\frac{p}{q}+,f}(\xi)\not \equiv 0$ and $P_{\frac{p}{q}-,f}(\xi)\not \equiv 0$. Denote by $\mathcal{O}$ the set of all $C^r$ exact
area-preserving  monotone twist diffeomorphisms that don't admit any homotopically non-trivial invariant circles with rotation number
$\frac{p}{q}$. We need to prove that $\mathcal{O}$ is open and dense. In fact, the denseness of $\mathcal{O}$ was proved by Mather
in \cite{MR849654} or \cite{MR1139556}, and the proof of openness is similar with Case $\textup{\uppercase\expandafter{\romannumeral 1}}$
by Theorem \ref{main theorem 1}.

\end{proof}

\begin{Rem}
Similar results also hold for analytic topology. Notice that in analytic situations, we need an additional condition on the irrational number $\omega$, i.e.,
$$\limsup\limits_{n\rightarrow+\infty}\frac{\log\log q_{n+1}}{\log q_n}>0,$$
where $(p_n/q_n)_{n\in\Z}$ is the sequence of approximants given by the continued fraction expansion of $\omega$ (see\cite{MR1279471}).
\end{Rem}

A set is called \emph{residual} if it is a countable intersection of open and dense sets.
\begin{Cor}
There exists a residual set $\mathcal{R}\subseteq \J^r$ in the $C^r$ topology which satisfies:
for each $f\in\mathcal{R}$, there is an open and dense set $\mathcal{H}(f)\subseteq \R$ such that for all $\omega\in \mathcal{H}(f)$, $f$ does not admit any homotopically non-trivial invariant circle with rotation number $\omega$.
\end{Cor}
\begin{proof}
Let $\{r_n\}_{n\in\Z}$ denotes the set of all rational numbers in $\R$. By Theorem \ref{application}, we obtain an open and dense set $\mathcal{O}_n$ such that for all $f\in\mathcal{O}_n$, there is no homotopically non-trivial invariant circle with rotation number $r_n$. We set $\mathcal{R}=\bigcap\limits_{n\in\Z}\mathcal{O}_n$, it is a residual set in $\J^r$ in the $C^r$ topology.

Take $g\in\mathcal{R}$, then we obtain that the Peierls barriers $P_{r_n+,g}(\xi)\not\equiv 0$ and $P_{r_n-,g}(\xi)\not\equiv 0$ for all $n$. We deduce from Proposition \ref{continuous of modulus} that $\omega \mapsto P_{\omega,g}$ is right-continuous at the rotation symbol $\omega=r_n+$ and left-continuous at the rotation symbol $\omega=r_n-$. Thus, there exists an open interval $(a_n, b_n)\ni r_n$ such that
$$P_{\omega,g}(\xi)\not\equiv 0,~~\textup{for all symbol}~~\omega~~\textup{with}~~\pi(\omega)\in(a_n, b_n).$$
If we set $\mathcal{H}(g)=\bigcup\limits_{n\in\Z}(a_n, b_n)$, then it's open and dense since $\{r_n\}_{n\in\Z}$ is dense in $\R$, and
$$P_{\omega,g}(\xi)\not\equiv 0,~~\textup{for all symbol}~~\omega~~\textup{with}~~\pi(\omega)\in\mathcal{H}(g),$$
which means that there is no homotopically non-trivial $g$-invariant circle with rotation number $\omega\in\mathcal{H}(g).$ This completes our proof.
\end{proof}

\section{Acknowledgments}
The authors were supported by
National Basic Research Program of China (973 Program)
(Grant No. 2013CB834100), National Natural Science Foundation of China (Grant No. 11631006, Grant No. 11201222)
and a program PAPD of Jiangsu Province, China.


\begin{thebibliography}{10}

\bibitem{MR719634}
S.~Aubry and P.-Y. Le~Daeron.
\newblock The discrete {F}renkel-{K}ontorova model and its extensions. i.
  {E}xact results for the ground-states.
\newblock {\em Phys. D}, 8(3):381--422, 1983.

\bibitem{MR945963}
V.~Bangert.
\newblock Mather sets for twist maps and geodesics on tori.
\newblock {\em Dynamics reported}, 1:1--56, 1988.

\bibitem{MR1384394}
V.~Bangert.
\newblock Geodesic rays, {B}usemann functions and monotone twist maps.
\newblock {\em Calc. Var. Partial Differential Equations}, 2(1):49--63, 1994.

\bibitem{MR3565373}
Q.~Chen and M.~Zhou.
\newblock Perturbation estimates of weak {KAM} solutions and minimal invariant
  sets for nearly integrable {H}amiltonian systems.
\newblock {\em Proc. Amer. Math. Soc.}, 145(1):201--214, 2017.

\bibitem{MR3061774}
C.-Q. Cheng and L.~Wang.
\newblock Destruction of {L}agrangian torus for positive definite {H}amiltonian
  systems.
\newblock {\em Geom. Funct. Anal.}, 23(3):848--866, 2013.

\bibitem{MR1279471}
G.~Forni.
\newblock Analytic destruction of invariant circles.
\newblock {\em Ergodic Theory Dynam. Systems}, 14(2):267--298, 1994.

\bibitem{MR1323222}
G.~Forni and J.~Mather.
\newblock {\em Action minimizing orbits in {H}amiltomian systems}, chapter~3,
  pages 92--186.
\newblock Springer Berlin Heidelberg, Berlin, Heidelberg, 1994.

\bibitem{MR728564}
M.-R. Herman.
\newblock Sur les courbes invariantes par les diff\'eomorphismes de l'anneau.
  vol. 1.
\newblock {\em Soci\'et\'e Math\'ematique de France, Paris}, 103:103--104,
  1983.

\bibitem{MR670747}
J.~Mather.
\newblock Existence of quasi-periodic orbits for twist homeomorphisms of the
  annulus.
\newblock {\em Topology}, 21(4):457--467, 1982.

\bibitem{MR849654}
J.~Mather.
\newblock A criterion for the non-existence of invariant circles.
\newblock {\em Inst. Hautes \'Etudes Sci. Publ. Math.}, 63(1):153--204, 1986.

\bibitem{MR920622}
J.~Mather.
\newblock {\em Modulus of continuity for {P}eierls's barrier}, chapter~18,
  pages 177--202.
\newblock Springer Netherlands, 1987.

\bibitem{MR967638}
J.~Mather.
\newblock Destruction of invariant circles.
\newblock {\em Ergodic Theory Dynam. Systems}, 8$^*$:199--214, 1988.

\bibitem{MR1139556}
J.~Mather.
\newblock Differentiability of the minimal average action as a function of the
  rotation number.
\newblock {\em Bol. Soc. Brasil. Mat. (N.S.)}, 21(1):59--70, 1990.

\bibitem{MR1109661}
J.~Mather.
\newblock Action minimizing invariant measures for positive definite
  {L}agrangian systems.
\newblock {\em Math. Z.}, 207(2):169--207, 1991.

\bibitem{MR1275203}
J.~Mather.
\newblock Variational construction of connecting orbits.
\newblock {\em Ann. Inst. Fourier (Grenoble)}, 43(5):1349--1386, 1993.

\bibitem{MR863203}
J.~Moser.
\newblock Monotone twist mappings and the calculus of variations.
\newblock {\em Ergodic Theory Dynam. Systems}, 6(3):401--413, 1986.

\end{thebibliography}
\end{document}